\documentclass[12pt]{amsart}
\usepackage{fullpage}
 
\usepackage[nofootnote,draft,silent]{fixme}
\usepackage{xypic}
\usepackage{enumerate}
\usepackage{amsthm}
\usepackage{color}
\usepackage{amsmath}
\usepackage{amssymb}
\usepackage{amsfonts}
\usepackage{hyperref}

\newtheorem{thm}{Theorem}[section]
\newtheorem{introthm}{Theorem}

\newtheorem*{thm*}{Theorem} 
\newtheorem{lem}[thm]{Lemma}

\newtheorem{prop}[thm]{Proposition}

\theoremstyle{definition}
\newtheorem{Def}[thm]{Definition}
{\newtheorem{ex}[thm]{Example}}
{\newtheorem{rem}[thm]{Remark}}
{}
{}

\newcommand{\C}{\ensuremath{\mathbb{C}}}
\newcommand{\Z}{\ensuremath{\mathbb{Z}}}
\newcommand{\Q}{\ensuremath{\mathbb{Q}}}

\newcommand{\A}{\ensuremath{\mathbb{A}}}

\newcommand{\GL}{\operatorname{GL}}

\newcommand{\Aut}{\operatorname{Aut}}
\newcommand{\Autd}{\underline{\operatorname{Aut}}^\del}

\newcommand{\Spec}{\operatorname{Spec}}

\newcommand{\Frac}{\ensuremath{\mathrm{Frac}}}

\newcommand{\mc}[1]{\mathcal #1}

\newcommand{\wh}{\widehat}

\newcommand{\del}{\ensuremath{\partial}}

\renewcommand{\P}{\mathbb{P}}

\title{Differential Embedding Problems over Laurent series fields}
\author{Annette Bachmayr, David Harbater and Julia Hartmann}

\date{February 5, 2018}

\begin{document}

\thanks{The first author was funded by the Deutsche Forschungsgemeinschaft (DFG) -
	grant MA6868/1-1 and by the Alexander von Humboldt foundation through a Feodor Lynen fellowship. The second and third authors were supported on NSF collaborative FRG grant: DMS-1463733; additional support was provided by NSF collaborative FRG grant DMS-1265290 (DH) and a Simons Fellowship (JH).
}

\begin{abstract}
We solve the inverse differential Galois problem over the fraction field of $k[[t,x]]$ and use this to solve split differential embedding problems over $k((t))(x)$ that are induced from $k(x)$.
The proofs use patching as well as prior results on inverse problems and embedding problems. 
\end{abstract}

\maketitle                                                                    

\section*{Introduction}
This paper is concerned with differential Galois theory over certain fields of Laurent series of characteristic zero. In particular, we study which linear algebraic groups occur as Galois groups of Picard-Vessiot rings over such fields and how these extensions fit together in towers. The latter is done by studying embedding problems, in analogy to usual Galois theory. The differential fields we consider are of the forms $k((x,t))$ and $k((t))(x)$ with derivation induced from $\frac{d}{dx}$.

Classically, differential Galois theory considered differential fields with algebraically closed fields of constants. More recently, there has been a lot of interest in extending results to more general differential fields (e.g.\ see \cite{amano-masuoka}, \cite{Andre}, \cite{BHH}, \cite{CHP}, \cite{Dyc}, \cite{LSP}). The Laurent series fields studied in this manuscript are examples of such fields. 

Differential embedding problems ask whether a given Picard-Vessiot ring with group $H$ can be embedded into one with group $G$, if $G$ is a given extension of $H$.  An embedding problem is {\em split} if there is a homomorphic section $H\rightarrow G$. 

The main goal of the current paper is to prove the following theorem (Theorem~\ref{main thm}):

\begin{introthm} \label{EP thm intro}
	Let $k$ be a field of characteristic zero and let $K=k((t))$. Consider the rational function field $F=K(x)$ with derivation $\del=\frac{d}{dx}$. Then every split $(K(x)/k(x))$-differential embedding problem over $F$ has a proper solution. 
\end{introthm}

As indicated in the statement of Theorem~\ref{EP thm intro}, we restrict to split $(K(x)/k(x))$-differential embedding problems; i.e., split embedding problems over $K(x)$ that are induced from split embedding problems over $k(x)$ (see the beginning of Section~\ref{sec: diffebp} for a precise definition).  The splitness assumption is generally necessary (see Remark~\ref{nonsplit} below). 

Note that the inverse differential Galois problem is contained in the question of whether split differential embedding problems have solutions, by taking $H$ above to be the trivial group. Over $k((t))(x)$, the inverse differential Galois problem was solved in \cite{BHH}: All linear algebraic groups over $k((t))$ occur as differential Galois groups of Picard-Vessiot rings of this field.
Differential embedding problems were previously studied in \cite{MatzatPut}, \cite{Hartmann}, \cite{Oberlies}, \cite{Ernst}; and it was shown in \cite{BHHW} that {\em every} differential embedding problem over $\C(x)$ has a proper solution. In the current paper, we build on results shown in \cite{BHHW} and in \cite{BHH}. The proofs in those papers used patching 
methods, which can be applied over the complex numbers as well as Laurent series fields because these fields are complete.

Theorem~\ref{EP thm intro} is a key ingredient in the companion paper \cite{BHHP}. There it is shown that all split differential embedding problems over $k((t))(x)$ (even if not induced) have proper solutions; and moreover this holds for $K(x)$ where $K$ is any {\em large} field of infinite transcendence degree over $\Q$. The latter class of fields was introduced by Pop and turned out to provide the appropriate context to study inverse Galois theory and other algebraic structures (see \cite{pop:littlesurvey} for the definition and a survey on large fields).  Laurent series fields are large, and every large field is existentially closed in its Laurent series field. Proving Theorem~\ref{EP thm intro} and then passing from Laurent series fields to large fields is analogous to (and was motivated by) what was done earlier in usual Galois theory, where it was shown that all finite split embedding problems over $K(x)$ have proper solutions if $K$ is large (see \cite{Po96}, \cite{HJ98}, and \cite{HS05}).

We also consider the inverse differential Galois problem over the fraction field $k((x,t))$ of $k[[x,t]]$, equipped with the derivation given by $\partial/\partial x$. We prove the following result (see Theorem~\ref{thm inverse problem}):

\begin{introthm} \label{IGP thm intro}
	Let $k$ be a field of characteristic zero and let $G$ be a linear algebraic group over $k((t))$. Then there is a Picard-Vessiot ring over $k((x,t))$ with differential Galois group~$G$.
\end{introthm}

This result is used in the proof of Theorem~\ref{EP thm intro}.  For comparison, not all linear algebraic groups over $k((t))$ occur as differential Galois groups over the iterated Laurent series field $k((t))((x))$, and not all such groups defined over $k$ occur over $k((x))$.  As a special case, Theorem~\ref{IGP thm intro} recaptures the fact that every finite group is a Galois group over $k((x,t))$ in characteristic zero, and shows moreover that there is such an extension in which $k((t))$ is algebraically closed. Our theorem more generally also shows that all nonconstant finite group schemes occur as differential Galois groups over $k((x,t))$.

In usual Galois theory, the inverse problem over $k((x,t))$ was solved in \cite[Corollary~3.21]{Le99}, using patching methods.  Later, in \cite[Theorem~5.1]{HS05}, it was shown that every finite split embedding problem over $k((x,t))$ has a proper solution. Whether the latter can be generalized to differential embedding problems is still an open question. 

The manuscript is organized as follows. Section~\ref{setup} 
explains a patching setup for function fields over Laurent series fields, and equips this setup with a differential structure. In Section~\ref{PVtheory}, we prove results on the existence of Picard-Vessiot rings that are contained in the types of fields that arise in the patching setup, over an arbitrary constant field of characteristic zero.
Section~\ref{sec inverse} solves the inverse problem over $k((x,t))$, by building on the earlier results and \cite{BHHW}. In Section~\ref{sec: diffebp} we define split $K(x)/k(x)$-differential embedding problems and show that these are solvable when $K=k((t))$.

\medskip

We thank Florian Pop for helpful discussions.

\section{A differential setup for patching}\label{setup}

We begin by recalling the setup of patching over fields from \cite{HH}, which we will apply in our situation.  	
Let $T$ be a complete discrete valuation ring with residue field $k$ and uniformizer $t$; let $\wh X$ be a smooth connected projective $T$-curve with closed fiber $X$; and let $F$ be the function field of $\wh X$.  To each closed point $P \in X$ we assign a field $F_P$ that contains $F$; viz., $F_P$ is the fraction field of the completion $\wh R_P$ of the local ring $R_P:=\mathcal{O}_{\wh X,P}$ of $\wh X$ at $P$.  Also, to 
each non-empty affine open subset $U \subset X$ we assign an overfield $F_U$ of $F$.  Here $F_U$ is the fraction field of the $t$-adic completion $\wh R_U$ of the subring $R_U \subset F$ of elements that are regular at the points of $U$.  For $P,U$ as above, the fields $F_P, F_U$ are each contained in a common overfield $F_P^\circ$, which is a complete discretely valued field.  Namely, we consider the completion $\wh R_P^\circ$ of the localization of $\wh R_P$ at its ideal $t\wh R_P$, and take $F_P^\circ$ to be its fraction field.  

In this paper, it will suffice to treat the case of the projective $x$-line over $k[[t]]$.  In that situation, if $P$ is the point $x=0$ and $U$ is its complement in the closed fiber $\P^1_k$, 
then $F_P$ is the fraction field $k((x,t))$ of $\wh R_P = k[[x,t]]$, and
$F_U$ is the fraction field of $k[x^{-1}][[t]]$.  The field $F_P^\circ$ is $k((x))((t)) = \Frac(k((x))[[t]])$, which indeed contains $F_P$ and $F_U$. (See also the discussion before \cite[Theorem~5.10]{HH}.)  More generally, we have the following:

\begin{prop} \label{prop patching}
Let $k$ be a field of characteristic zero, let $T=k[[t]]$, and let $\wh X = \P^1_T$, with closed fiber $X = \P^1_k$. 
Let $P$ be a closed point of $\A^1_k\subseteq X$ that is not necessarily $k$-rational, and let $U$ be the complement of $P$ in $X$.  Let $q:=q(x) \in k[x] = \mc O(\A^1_k)$ generate the maximal ideal associated to $P$, let $r=\deg(q)$, and let $k'=k[x]/(q)$.  Then $F_P=k'((q,t))$, $F_P^\circ=k'((q))((t))$,
and $F_U=\Frac(k[q^{-1}, xq^{-1},\dots,x^{r-1}q^{-1}][[t]])$.
\end{prop}

\begin{proof} 
First note that the $(q)$-adic completion of $k[x]_{(q)}$ is a complete discrete 
valuation ring that is isomorphic as a $k[x]$-algebra to $k'[[q]]$ 
(\cite[II.4, Propositions~5,6]{Serre:LF}).  Thus $\wh R_P = k'[[q,t]]$, and the assertions about $F_P$ and $F_P^\circ$ follow.  The coordinate ring of $U$ is the 
subring of $k[x,q^{-1}]$ consisting of elements with no poles at $x=\infty$, and this is just $k[q^{-1}, xq^{-1},\dots,x^{r-1}q^{-1}]$.  So $\wh R_U = k[q^{-1}, xq^{-1},\dots,x^{r-1}q^{-1}][[t]]$, and $F_U$ is as asserted.
\end{proof}

The $k$-linear embedding $k[x] \hookrightarrow k'[[q]]$ given in the above proof
extends to a $k$-linear embedding $k(x) \hookrightarrow k'((q))$ between their fraction fields.
It follows that $F_U\subseteq 
k(x)((t))$ embeds into $F_P^\circ$ and hence there are inclusions 
$F_P,F_U\subseteq F_P^\circ$ as predicted.

We want to use fields of the above type in a differential context. 
To do that we  need to extend given derivations from rational function fields to the fields that occur in Proposition \ref{prop patching}. We start with a preliminary lemma. 

\begin{lem}\label{prop derivations}
Let $k$ be a field of characteristic zero and consider the iterated Laurent series field $F_0=k((w))((t))$. Then:
\begin{enumerate}
\item \label{del_a} For each $a\in F_0^\times$,
$$\del_a\colon F_0\to F_0, \ \sum \limits_{j=n}^\infty\sum\limits_{i=n_j}^\infty \alpha_{ij}w^it^j\mapsto a\cdot  \sum \limits_{j=n}^\infty\sum\limits_{i=n_j}^\infty i\alpha_{ij}w^{i-1}t^j$$ 
defines a derivation on $F_0$ with field of constants $k((t))$, extending $a\frac{\del}{\del w}$ on $k((w))$. 
\item\label{res1} If $a\in k((w,t))$, then $\del_a$ restricts to a derivation $\del_a\colon k((w,t))\to k((w,t))$; i.e., $\del_a(k((w,t)))\subseteq k((w,t))$.
\item \label{res3} If $B$ is a $k_0$-subalgebra of $k((w))$ for some subfield $k_0$ of $k$ and if $(b\cdot\frac{\del}{\del w})(B) \subseteq B$ for some non-zero $b\in k((w))$ with $a/b \in \Frac(B[[t]])$, then $\del_a$ restricts to a derivation on $\Frac(B[[t]])$.
\end{enumerate}
\end{lem}
\begin{proof}
\ref{del_a} This is immediate from the fact that $\frac{\del}{\del w}\colon k((w))\to k((w)), \  \sum\limits_{i=n}^\infty \alpha_{i}w^i\mapsto\sum\limits_{i=n}^\infty i\alpha_{i}w^{i-1}$ defines a derivation on $k((w))$ with field of constants $k$. 

\ref{res1} Since $\frac{\del}{\del w}$ restricts to a derivation $k[[w]]\to k[[w]]$, $\del_a$ restricts to a derivation $k[[w,t]]\to k((w,t))$ if $a\in k((w,t))$. Thus $\del_a$ restricts to a derivation on the fraction field $k((w,t))$ of $k[[w,t]]$. 

\ref{res3} Under the hypothesis of ~\ref{res3}, let $f=\sum _{i=0}^\infty b_it^i \in B[[t]]$. Then 
$$\del_a(f)=a\sum \limits_{i=0}^\infty\frac{\del}{\del w}(b_i)t^i=\frac{a}{b}\sum \limits_{i=0}^\infty(b\frac{\del}{\del w})(b_i)t^i \in \Frac(B[[t]])$$ 
and hence $\del_a$ restricts to a derivation on $\Frac(B[[t]])$. 
\end{proof}

\begin{Def}
A \textit{diamond with the factorization property} is a quadruple $(F,F_1,F_2,F_0)$ of fields with inclusions $F\subseteq F_1,F_2$ and $F_1,F_2\subseteq F_0$ such that $F_1\cap F_2=F$ and such that every matrix $A\in \GL_n(F_0)$ can be written as a product $A=B\cdot C$ with matrices $B\in \GL_n(F_2)$ and $C\in \GL_n(F_1)$. If in addition $\operatorname{char}(F)=0$ and all of these fields are equipped with a derivation that is compatible with the inclusions, then $(F,F_1,F_2,F_0)$ is called a \textit{differential diamond with the factorization property}. Note that we do \textit{not} require that $F, F_1,F_2,F_0$ all have the same field of constants. 
\end{Def}

\begin{prop}\label{prop diff diamond}
In the situation of Proposition \ref{prop patching}, 
let $a\in F^\times$ and
equip $F_P^\circ=k'((q))((t))$ with the derivation $\del_a$ as defined in Lemma~\ref{prop derivations}\ref{del_a} (taking $w=q$). Then $\del_a$ restricts to derivations on $F_P$ and $F_U$ and it restricts to the derivation $\beta\cdot\del/\del x$ on $F=k((t))(x)$, with $\beta=\left(\frac{\del q}{\del x}\right)^{-1}a $. With respect to these derivations, $(F,F_P,F_U,F_P^\circ)$ is a differential diamond with the factorization property.   
\end{prop}

\begin{proof}
By \cite[Proposition~6.3]{HH}, $F_P\cap F_U=F$ as subfields of $F_P^\circ$.  Moreover, by \cite[Theorem~5.10]{HH}, every matrix $A\in \GL_n(F_P^\circ)$ can be written as a product $A=B\cdot C$, for some matrices $B\in \GL_n(F_P)$ and $C\in \GL_n(F_U)$.  So $(F,F_P,F_U,F_P^\circ)$ is a diamond with the factorization property, and the second assertion follows from the first. 

For the first assertion, note that 
the derivation $\del_a$ on $F_P^\circ$ and the derivation $\beta\cdot \del/\del x$ on $F$ are each constant on $k((t))$, and each maps $q$ to $a$; hence
$\del_a$ and $\beta\cdot\del/\del x$ restrict to the same derivation $k((t))(q)\to F$.
As $F=k((t))(x)$ is a finite separable extension of $k((t))(q)$, it follows that the derivation $\del_a$ on $F_P^\circ$ restricts to $\beta\cdot\del/\del x$ on $F$.
Moreover, $\del_a$ restricts to a derivation on $F_P=k'((q,t))$ by Lemma~\ref{prop derivations}\ref{res1}, again denoted by $\del_a$.
It is given by (the restriction of) the formula in Lemma~\ref{prop derivations}\ref{del_a}.

To conclude the proof, it remains to show that $\del_a$ restricts to a derivation on $F_U=\Frac(k[q^{-1}, xq^{-1},\dots,x^{r-1}q^{-1}][[t]])$.  
First note that $k(q)\subseteq k(x)\subseteq k'((q))$, where the last inclusion is as in the comment after Proposition~\ref{prop patching} (and is $k$-linear). Set $B=k\!\left[\frac{1}{q},\frac{x}{q},\dots,\frac{x^{r-1}}{q}\right]$. Then $B$ is a $k$-subalgebra  of $k(x)$ and thus it is a $k$-subalgebra of $k'((q))$.  Set $b=\frac{\del q}{\del x} \in k(x)$. Then $b\in k'((q))$ and $a/b=\beta\in F\subseteq F_U=\Frac(B[[t]])$. In order to apply Lemma~\ref{prop derivations}\ref{res3} (with $w=q$) we next check that $(b\frac{\del}{\del q})(B) \subseteq B$. As $\frac{\del}{\del x}(q)=b=(b\frac{\del}{\del q})(q)$, the derivation $\frac{\del}{\del x}$ on $k(x)$ and the derivation $b\frac{\del}{\del q}$ on $k'((q))$ both restrict to the same derivation $k(q)\to k(x)$. Using that $k(x)$  is a finite separable extension of $k(q)$, we conclude that $b\frac{\del}{\del q}$ restricts to $\frac{\del}{\del x}$ on $k(x)$. Hence for all $0\leq i < r$, $(b\frac{\del}{\del q})(\frac{x^i}{q})=\frac{ix^{i-1}}{q}-\frac{x^i}{q}\cdot \frac{\frac{\del}{\del x}(q)}{q}  \in B$ and we conclude that $(b\frac{\del}{\del q})(B) \subseteq B$. Hence $\del_a$ restricts to a derivation on $F_U$ by Lemma~\ref{prop derivations}\ref{res3}, as needed. 
\end{proof}

\begin{ex}\label{ex diff}
Let $k$ be a field of characteristic zero and $F=k((t))(x)$, $F_1=\Frac(k[x^{-1}][[t]])$, $F_2=k((x,t))$ and $F_0=k((x))((t))$. Then $\frac{\del}{\del x}$ canonically extends from $F$ to $F_1,F_2,F_0$ and $(F,F_1,F_2,F_0)$ is a differential diamond with the factorization property. Indeed, this is a special case of Proposition \ref{prop diff diamond} with $q=x$, i.e. $P$ is the origin on $\mathbb P^1_k$, and $a=\beta=1$.
\end{ex}

\section{Picard-Vessiot theory}\label{PVtheory}

Although classical differential Galois theory (e.g., as in \cite{vdPutSinger}) assumes that the field of constants in a differential field is algebraically closed, we will allow more general fields of constants, provided that the characteristic is zero (see \cite{Dyc} for details).

If $R$ is a differential ring, let $C_R$ denote its ring of constants. Let $F$ be a differential field with field of constants $K=C_F$ (which might not be algebraically closed). Consider a matrix $A\in F^{n\times n}$. A \textit{Picard-Vessiot ring} $R/F$ for the differential equation $\del(y)=Ay$ is a differential ring extension $R/F$ that satisfies the following conditions: there exists a matrix $Y\in \GL_n(R)$ with $\del(Y)=A\cdot Y$ and such that $R=F[Y,\det(Y)^{-1}]$ (i.e., $R$ is generated as an $F$-algebra by the entries of $Y$ and $\det(Y)^{-1}$), and such that $R$ is a simple differential ring with $C_R=C_F$. 
By differential simplicity, every Picard-Vessiot ring $R/F$ is an integral domain; moreover, 
$C_{\Frac(R)}=C_F$.  If $C_F$ is algebraically closed, then for every $A$ there is a unique Picard-Vessiot ring up to isomorphism (\cite[Prop.~1.18]{vdPutSinger}); but this is not the case for general fields of constants. 

Given a Picard-Vessiot ring $R/F$, for each $K$-algebra $S$ let $G(S)=\Aut^\del(R\otimes_K S/F\otimes_K S)$ be the group of differential $(F\otimes_K S)$-automorphisms of $R\otimes_K S$, where we consider $S$ as a differential ring with the trivial derivation.  Then $G$ is a functor $\Autd(R/F)$ from the category of $K$-algebras to the category of groups.  In fact, this functor  is a linear algebraic group, the {\em differential Galois group} of the Picard-Vessiot ring. By construction, its group of $K$-rational points is $\Aut^\del(R/F)$.  (When $K$ is algebraically closed, one commonly identifies $G$ with $\Aut^\del(R/F)$, since a $K$-scheme is determined by its $K$-points.)

If $G$ is a linear algebraic group over $K$ and $K'/K$ is a field extension, we let $G_{K'}$ denote the base change of $G$ from $K$ to $K'$. If $R/F$ is a Picard-Vessiot ring with differential Galois group $G$ and $K'/K$ is an algebraic extension of constants, then $R\otimes_K K'$ is a Picard-Vessiot ring over the field $F\otimes_K K'$ 
with differential Galois group $G_{K'}$ (here simplicity follows from \cite[Cor. 2.7]{Dyc}).

In order to apply patching to differential Galois theory, we want to embed Picard-Vessiot rings into Laurent series fields.  We will use the following:

\begin{prop}\label{prop embed pvr}
Let $F=k(x)$ be a rational function field over $k$ with derivation $\del=\frac{d}{dx}$ and let $R/F$ be a Picard-Vessiot ring. Then there exists an element $\alpha\in k$ and a finite extension $k'/k$ such that $R$ embeds into $k'((x-\alpha))$ as a differential $k(x)$-algebra. 	
\end{prop}

\begin{proof}
Suppose $R$ is a Picard-Vessiot ring for the differential equation 
$\partial (y) = Ay$ for some matrix
$A\in F^{n\times n}$. Let $\alpha \in k$ be a regular point of the differential equation (i.e., the 
entries of $A$ have no poles at $\alpha$). Then there is a fundamental 
solution matrix $Y$ with
entries in $k[[x-\alpha]]$; the coefficients of the power series can be 
found recursively (see, e.g., the proof of [Dyc08, Theorem 2.8]). Let 
$R_0=F[Y,\det(Y)^{-1}]\subseteq k((x-\alpha))$. Thus
$C_{\Frac(R_0)}=k=C_F$; so it follows from [Dyc08, Corollary 2.7] that 
$R_0$ is simple and hence is a Picard-Vessiot ring for the differential 
equation $\partial(y)=Ay$.

Both $R\otimes_k \overline{k}$ and $R_0\otimes_k \overline{k}$ are Picard-Vessiot rings over $\overline{k}(x)$ for the same differential equation; hence they are isomorphic over $\overline{k}(x)$. Thus there is a finite extension $k'/k$ such that $R\otimes_k k'$ and $R_0\otimes_k k'$ are isomorphic over $k'(x)$. The inclusion $R_0\subseteq k((x-\alpha))$ yields a canonical differential homomorphism $R_0\otimes_k k' \to   k'((x-\alpha))$ which is injective since $R_0\otimes_k k'$ is a simple differential ring. So we obtain an embedding $R \hookrightarrow R\otimes_k k'\cong R_0\otimes_k k' \hookrightarrow k'((x-\alpha))$.
\end{proof}

The next lemma shows that in constructing Picard-Vessiot rings, one may modify the derivation by an element in the ground field. 

\begin{lem}\label{lem scale derivation}
Let $F_0$ be a field with a derivation $\del$, let $F$ be a differential subfield of $F_0$, and let $f \in F^\times$. If $(R,\del)$ is a Picard-Vessiot ring over $(F,\del)$ with differential Galois group $G$, then $\del':=f\del$ is a derivation on $R$ and $(R,\del')$ is a Picard-Vessiot ring over $(F,\del')$ with differential Galois group $G$. If moreover there is an $F$-linear differential embedding $(R,\del)\hookrightarrow (F_0,\del)$ then there is an $F$-linear differential embedding $(R,\del')\hookrightarrow (F_0,\del')$. 
\end{lem} 

\begin{proof}
The first assertion was shown in \cite[Lemma~4.2]{BHH} and its proof.  For the second assertion, 
if $\phi\colon R\to F_0$ is an injective ring homomorphism that restricts to the identity on $F$ and commutes with $\del$, then $\phi(f)=f$, and so $\phi$ also commutes with $\del'$.
\end{proof}

The following theorem asserts the existence of a Picard-Vessiot ring $R$ over the field $k((t))(z)$ with specified differential Galois group, 
such that $R$ can be embedded into $k((t,z))$ as a differential subring.

\begin{thm}\label{thm inverse problem rational function field}
Let $k$ be a field of characteristic zero, and let $F=k((t))(z)$ be a rational function field over $k((t))$ equipped with the derivation $\del=a\cdot \frac{d}{dz}$ for a fixed non-zero 
$a \in k((t))(z)$.  Let $G$ be a linear algebraic group defined over $k((t))$.  Then there is a 
Picard-Vessiot ring $R/F$ with differential Galois group $G$ and an $F$-linear differential embedding $R \hookrightarrow k((t,z))$, where $k((t,z))$ is equipped with the derivation $\del_a$ as in Lemma~\ref{prop derivations}\ref{res1}.
\end{thm}

\begin{proof}
Theorem~4.5 of \cite{BHH} shows that there is a Picard-Vessiot ring
$R/k((t))(z)$ with differential Galois group $G$.  The proof there used
the derivation $\del=a\cdot \frac{d}{dz}$ with $a=t^{-1}$.  In that
situation, the proof of Corollary~4.6 of loc.\ cit.\ then shows that $R$
can be chosen to be contained in $k((t,z))$.

For general $a$, define the derivation $\del'=\frac{1}{ta}\del=t^{-1}\cdot \frac{d}{dz}$ on $F$. By the previous paragraph, there is a Picard-Vessiot ring $(R,\del')$ over $(F,\del')$ with differential Galois group $G$ and an $F$-linear differential embedding $(R,\del')\hookrightarrow (k((t,z)),\del_{t^{-1}})$. By Lemma \ref{lem scale derivation}, 
$\del:=ta\cdot \del'$ defines a derivation on $R$, and $(R,\del)$ is a Picard-Vessiot ring over $(F,\del)$ with differential Galois group $G$. Moreover, there is an $F$-linear differential embedding
$$(R,\del)\hookrightarrow  (k((t,z)),(ta)\cdot\del_{t^{-1}}).$$
 The assertion now follows from the equality $(ta)\cdot\del_{t^{-1}}=\del_a$. 
\end{proof}

\section{The inverse differential Galois problem over $k((x,t))$}\label{sec inverse}

In this section, we solve the inverse differential Galois problem over $k((x,t))$
with respect to the derivation $\del/\del x$.  More precisely, we equip the field $k((x))((t))$ with the derivation 
$\del=\del_1$ defined in Lemma~\ref{prop derivations}\ref{del_a} (with $w=x$); and we also write $\del$ for the restriction of this derivation to $k((x,t))$, as in Lemma~\ref{prop derivations}\ref{res1}.  These derivations extend $\del/\del x$ on $F=k((t))(x)$.

The following lemma from \cite{BHHW} allows us to lift certain Picard-Vessiot rings $R/F$ to  Picard-Vessiot rings over a differential field $F_2\supseteq F$, while preserving the differential Galois group. This will be a major ingredient in the proof of Theorem~\ref{thm inverse problem} below.

\begin{lem}[\cite{BHHW}, Lemma~2.9]\label{lem lifting PVR}
Let $F\subseteq F_1,F_2\subseteq F_0$ be differential fields of characteristic zero such that $F_1\cap F_2=F$ and $C_{F_0}=C_{F}$. Assume that $R/F$ is a Picard-Vessiot ring such that $R$ is a differential $F$-subalgebra of $F_1$. Let $G$ be its differential Galois group. Then the compositum $F_2R\subseteq F_0$ is a Picard-Vessiot ring over $F_2$ with differential Galois group $G$.
\end{lem}

\begin{lem} \label{lem emb}
	Define $\Phi:k[[z,t]]\to k[x^{-1}][[t]]$ by $$\Phi(f):=\sum \limits_{k=0}^\infty\left(\sum\limits_{i=0}^k\alpha_{i,k-i}x^{-i} \right)t^k \text{ for } f=\sum \limits_{i,j=0}^\infty\alpha_{ij}z^it^j \in k[[z,t]].$$ Then $\Phi$ defines an injective ring homomorphism which maps $z$ to $t/x$. It extends uniquely to an embedding $k((z,t))\hookrightarrow \Frac(k[x^{-1}][[t]])$.  
This embedding is a $k((t))(z)$-algebra homomorphism, where we regard $\Frac(k[x^{-1}][[t]])$
as a $k((t))(z)$-algebra via $z \mapsto t/x$.
\end{lem}

\begin{proof}
Obviously, $\Phi$ is injective and additive. To see that it is multiplicative, take $f=\sum \limits_{i,j=0}^\infty\alpha_{ij}z^it^j$ and $g=\sum \limits_{i,j=0}^\infty\beta_{ij}z^it^j \in k[[z,t]]$. Then 
$$\Phi(fg)=\sum\limits_{k=0}^\infty\left(\sum\limits_{i=0}^k\sum \limits_{i_1=0}^i\sum \limits_{j_1=0}^{k-i}\alpha_{i_1,j_1}\beta_{i-i_1,k-i-j_1}x^{-i} \right)t^k,$$ whereas 
$$\Phi( f)\cdot \Phi(g)=\sum\limits_{k=0}^\infty\left(\sum\limits_{k_1=0}^k\sum \limits_{i_1=0}^{k_1}\sum \limits_{i_2=0}^{k-k_1}\alpha_{i_1,k_1-i_1}\beta_{i_2,k-k_1-i_2}x^{-i_1-i_2} \right)t^k $$ and it is easy to check that these two expressions coincide (via $i=i_1+i_2$ and $k_1=i_1+j_1$). The last two assertions are then immediate.
\end{proof}

\begin{thm}\label{thm inverse problem}
Let $k$ be a field of characteristic zero and let $G$ be a linear algebraic group over $k((t))$. View $k((x,t))\subseteq k((x))((t))$ as differential fields as explained above. Then there exists a Picard-Vessiot ring over $k((x,t))$ with differential Galois group $G$. Moreover, this Picard-Vessiot ring can be constructed in such a way that it embeds into $k((x))((t))$ as a differential $k((x,t))$-subalgebra.
\end{thm}

\begin{proof}
Define $F=k((t))(x)$, $F_1=\Frac(k[x^{-1}][[t]])$, $F_2=k((x,t))$, $F_0=k((x))((t))$. Then $F\subseteq F_1,F_2\subseteq F_0$ is a differential diamond by Example~\ref{ex diff}. 
In particular, $F_1\cap F_2=F$.
Moreover, $C_{F_0}=C_F$ by Lemma~\ref{prop derivations}\ref{del_a}. Our goal is to construct a Picard-Vessiot ring over $F_2$ with differential Galois group $G$ using Lemma~\ref{lem lifting PVR}.

Set $z=t/x\in F$ and $a=\del(z)={-z^2}/{t}$. Then $F=k((t))(z)$ and the derivation $\del$ on $F$ equals $a\frac{d}{dz}$. Also, $F\subseteq k((t,z))\subseteq k((z))((t))$ and these are inclusions of differential fields, where we equip $k((z))((t))$ with the derivation $\del_a$ as in Lemma~\ref{prop derivations}\ref{del_a} (applied to $w=z$). Note that $\del_a$ restricts to $k((z,t))$ by Lemma~\ref{prop derivations}\ref{res1}. 

By Theorem \ref{thm inverse problem rational function field} we obtain a Picard-Vessiot ring $R/F$ with differential Galois group $G$ together with an embedding $R  \hookrightarrow k((z,t))$ of  differential $F$-algebras.

Lemma~\ref{lem emb} provides an injective $F$-algebra homomorphism $\Phi:k((z,t))\hookrightarrow F_1$. We claim that this is a differential embedding.  It suffices to show that $\del(\Phi(f))=\Phi(\del_a(f))$ for all $f\in k[[z,t]]$. Let $f\in k[[z,t]]$. Then $f$ can be written as $f=\sum\limits_{i,j=0}^\infty\alpha_{ij}z^it^j$ and 
$$\del_a(f)=a\sum\limits_{i,j=0}^\infty i\alpha_{ij}z^{i-1}t^j=-\frac{z}{t}\sum\limits_{i,j=0}^\infty i\alpha_{ij}z^it^j.$$ Hence 
$$\Phi(\del_a(f))=-\frac{1}{x}\sum\limits_{k=0}^\infty\left(\sum\limits_{i=0}^k i\alpha_{i,k-i}x^{-i}\right)t^k=\sum\limits_{k=0}^\infty\left(\sum\limits_{i=0}^k (-i)\alpha_{i,k-i}x^{-i-1}\right)t^k,$$ and the latter equals $\del(\Phi( f))$, as claimed. 

We conclude that $R$ embeds into $F_1$ as a differential $F$-algebra, 
and hence also into $F_0=k((x))((t))$. Therefore, the compositum $RF_2$ of $R$ and $F_2$ in $F_0$ is a differential $F_2$-subalgebra of $F_0$.  By Lemma \ref{lem lifting PVR}, it is a
Picard-Vessiot ring over $F_2$ with differential Galois group $G$. 
\end{proof}

\begin{rem} \label{remark step 3}
The conclusion of Theorem~\ref{thm inverse problem} remains true
if for some $a\in k((x,t))$ we instead equip $k((x,t))$ and $k((x))((t))$ with the derivation $\del_a$ as defined in Lemma \ref{prop derivations} (applied to $w=x$).  This follows from Lemma~\ref{lem scale derivation}, because $\del_a=a\cdot \del_1$.
\end{rem}

\section{Differential embedding problems }\label{sec: diffebp}
In this section we solve split embedding problems over $k((t))(x)$ which are induced from $k(x)$ (Theorem~\ref{main thm} below). In the related manuscript \cite{BHHP}, it is shown that this implies the solvability of all split differential embedding problems over $K(x)$, when $K$ is any large field of infinite transcendence degree over $\Q$. In particular, one obtains the solvability of all split differential embedding problems over $k((t)(x)$ for any field $k$ of characteristic zero.

Let $F$ be a differential field with field of constants $K$. A \textit{differential embedding problem over $F$} consists of an exact sequence $1\to N \to G \to H \to 1$ of linear algebraic groups over $K$ and a Picard-Vessiot ring $R/F$ with differential Galois group $H$. If the exact sequence splits, i.e., if $G\cong N\rtimes H$, the differential embedding problem is called \textit{split}. In this case, we write $(N\rtimes H,R)$ for the embedding problem. In the special case that $F=K(x)$ with derivation $\frac{d}{dx}$, we call a split differential embedding problem $(N\rtimes H,R)$ a \textit{split $(K(x)/k(x))$-differential embedding problem} for some subfield $k\subseteq K$, if it is {\em induced} from a split differential embedding problem $(N_0\rtimes H_0,R_0)$ over $k(x)$. That is, the base changes of linear algebraic groups $N_0, H_0, G_0$ from $k$ to $K$ are isomorphic to $N,H,G$, respectively; these isomorphisms are compatible with the structure of the semidirect product; and $R_0\otimes_{k(x)} K(x)\cong R$ as differential $K(x)$-algebras.

A \textit{proper solution} to a differential embedding problem consists of a Picard-Vessiot ring $S/F$ with differential Galois group $G$ together with a differential embedding $R\hookrightarrow S$ such that the following diagram commutes:
\[\xymatrix{ G  \ar@{->}[d]^\cong \ar@{->>}[rr]  && H \ar@{->}[d]_\cong \\
	\underline{\Aut}^\del(S/F) \ar@{->>}[rr]^{\operatorname{res}}&    & \underline{\Aut}^\del(R/F)} \]

To obtain our main theorem we will use the next result, which appeared in \cite{BHHW} (Theorem~2.14 there), and which was proven by patching differential torsors.

\begin{thm}\label{thm BHHW}
Let $(F,F_1,F_2,F_0)$ be a differential diamond with the factorization property and let $(N\rtimes H, R)$ be a split differential embedding problem over $F$ with the property that $R$ is a differential $F$-subalgebra of $F_1$. Assume further that there is a Picard-Vessiot ring $R_1/F_1$ with differential Galois group $N_{C_{F_1}}$, such that $R_1$ is a differential $F_1$-subalgebra of $F_0$. Then there exists a proper solution to the differential embedding problem $(N\rtimes H, R)$ over $F$.
\end{thm}

We are now in a position to prove our main theorem, using Theorem~\ref{thm BHHW}:

\begin{thm}\label{main thm}
Let $k$ be a field of characteristic zero and let $K=k((t))$. Consider the rational function field $F=K(x)$ with derivation $\del=\frac{d}{dx}$. Then every split $(K(x)/k(x))$-differential embedding problem over $F$ has a proper solution. 
\end{thm}

\begin{proof}
Let $(N\rtimes H, R)$ be a split $(K(x)/k(x))$-differential embedding problem. 
Thus $R$ is a Picard-Vessiot ring over $F$ with differential Galois group $H$.  Moreover,
there exist linear algebraic groups $N_0$ and $H_0$ over $k$ such that $N$ is the base change of $N_0$ and $H$ is the base change of $H_0$ from $k$ to $K$, and there exists a Picard-Vessiot ring $R_0$ over $k(x)$ with differential Galois group $H_0$ such that $R\cong R_0\otimes_{k(x)}K(x)$ as differential $K(x)$-algebras. 

By Proposition~\ref{prop embed pvr}, there exists an $\alpha\in k$ and a finite extension $k'/k$ such that $R_0$ embeds into $k'((x-\alpha))$ as a differential $k(x)$-algebra, where $k'((x-\alpha))$ is equipped with the derivation $\frac{\del}{\del(x-\alpha)}$. Let $q(Z) \in k[Z]$ be an irreducible and monic polynomial of degree $r \ge 1$ such that $k'\cong k[Z]/(q(Z))$. Define 
$$y=\frac{x-\alpha}{t}.$$
Then $K(y)=F$ and we use the method of patching over fields (see Section~\ref{setup}) over the $y$-line $\P^1_{k[[t]]}$, which we call $\wh X$. Let $P$ be the point $(q(y),t)$ on the closed fibre $X$ of $\wh X$, and set $U=X\smallsetminus\{P\}$.  Applying Proposition~\ref{prop patching} (with $y$ here playing the role of $x$ there), and abbreviating $q(y)$ to $q$,  
we have equalities
\begin{eqnarray*}
F_U&=&\Frac\left(k\!\left[\frac{1}{q},\frac{y}{q},\dots,\frac{y^{r-1}}{q}\right]\![[t]]\right) \\
F_P&=&k'((q,t)) \\
F_P^\circ&=&k'((q))((t)).
\end{eqnarray*}

To prove the theorem, it suffices to verify the hypotheses of Theorem~\ref{thm BHHW} for the fields $F_1=F_P$, $F_2=F_U$, $F_0=F_P^\circ$, with respect to compatible derivations on those fields that extend the derivation $\del$ on $F$.   First note that for any $a \in F^\times$,  
the quadruple $(F,F_P,F_U,F_P^\circ)$ is a differential diamond with the factorization property by Proposition~\ref{prop diff diamond}, with respect to the restrictions of $\del_a$ (again with $y$ here playing the role of $x$ there).  By choosing $a=\del(q) = \del q/\del x \in F^\times$, the restriction of $\del_a$ to $F$ is equal to $\del$, since that restriction is 
\[\left(\frac{\del q}{\del y}\right)^{-1}\cdot a\cdot \frac{\del}{\del y}=\left(\frac{\del y}{\del x}\right)\cdot \frac{\del}{\del y}=\frac{\del}{\del x}=\del\]
by Proposition~\ref{prop diff diamond}.
Next, by Theorem~\ref{thm inverse problem} and Remark~\ref{remark step 3} (with $q$ here playing the role of $x$ in Theorem~\ref{thm inverse problem}), there is
a Picard-Vessiot ring $R_1/F_1$ with differential Galois group $N_{k'((t))}$ such that $R_1$ is a differential $F_1$-subalgebra of $F_0$ with respect to $\del_a$.  Thus it remains to show that 
the Picard-Vessiot ring $R$ over $F$ is a differential $F$-subalgebra of $F_P$.

Recall that $R_0$ embeds into $k'((x-\alpha))$ as a differential $k(x)$-algebra; in particular, $R_0$ embeds into $k'((x-\alpha,t))$ as a differential $k(x)$-algebra, as does $F = k((t))(x) = k((t))(x-\alpha)$.  So there is a canonical differential $F$-algebra homomorphism $R\cong R_0\otimes_{k(x)} k((t))(x)\to k'((x-\alpha,t))$; and since $R$ is a simple differential ring, this homomorphism is injective. Hence we may consider $R$ as a differential $F$-subalgebra of $k'((x-\alpha,t))$, with respect to the given derivation on $R$ and the derivation $\frac{\del}{\del(x-\alpha)}$ on $k'((x-\alpha,t))$ that extends $\frac{\del}{\del(x-\alpha)}$ on $k'[[x-a,t]]$. 
It therefore suffices to show that there is an $F$-algebra inclusion 
of $k'((x-\alpha,t))$ into $F_P$ with respect to which 
the derivation $\del_a$ on $F_P$ restricts to the derivation $\frac{\del}{\del(x-\alpha)}$ on $k'((x-\alpha,t))$.

To define the desired inclusion $k'((x-\alpha,t)) \subset F_P$, first note that 
$k'[[x-\alpha,t]]$ is a subring of $k'[y][[t]]$, where as above
$y = (x - \alpha)/t$, since $\sum_{i,j=0}^\infty c_{i,j}(x-\alpha)^i t^j
= \sum_{i,j=0}^\infty c_{i,j}y^i t^{i+j}$.  Meanwhile, the $k$-linear
inclusion $k[y] \hookrightarrow k'[[q]]$ (as in the proof of Proposition~\ref{prop patching}, 
with $y$ here playing the role of $x$ there) extends to a $k'$-linear 
inclusion of $k'[y]$ into $k'[[q]]$ and then to a $k'[[t]]$-linear inclusion of
$k'[y][[t]]$ into $k'[[q,t]]$.  Passing to the fraction fields, we
obtain inclusions $k'((x-\alpha,t)) \subset \Frac\left(k'[y][[t]]\right)
\subset k'((q,t)) = F_P$ of $F$-algebras.

By Lemma~\ref{prop derivations} (with $w=q$), 
the derivation $\del_a$ on $F_P = k'((q,t))$
takes the element $\sum_j f_j(y)t^j \in k'[y][[t]]$ to $a\sum_j\frac{\del}{\del q}f_j(y)t^j$. Using $a=\del q/\del x=t^{-1}\del q/\del y$, we conclude that $\del_a$ takes $\sum_j f_j(y)t^j \in k'[y][[t]]$ to
\[t^{-1}\frac{\del q}{\del y}\sum_j\frac{\del}{\del q}f_j(y)t^j=t^{-1}\sum_j\frac{\del}{\del y}f_j(y)t^j
= \sum_j \frac{\del f_j(y)}{\del y}t^{j-1}.\]
Hence it takes $\sum_{i,j=0}^\infty c_{i,j}(x-\alpha)^i t^j \in k'[[x-\alpha,t]]$ to
$\sum_{i,j=0}^\infty c_{i,j} i(x-\alpha)^{i-1} t^j$, using that $y = (x - \alpha)/t$.
That is, $\del_a$ restricts to $\frac{\del}{\del(x-\alpha)}$ on $k'[[x-\alpha,t]]$,
and hence on $k'((x-\alpha,t))$.
\end{proof}

The strategy of the above proof can be understood geometrically as 
follows:  The given Picard-Vessiot ring $R_0$ embeds into a field of the form 
$k'((x-\alpha))$, for some $\alpha\in k$ and some finite extension $k'/k$.  Hence $R$ 
embeds into the field $k'((x-\alpha,t))$, 
which can be identified with $F_{P_0} \otimes_k k'$, where $P_0$ is the 
closed point $(x-\alpha,t)$ on the projective $x$-line $\wh X_0 := 
\P^1_{k[[t]]}$.  Let $f:\wh Y \to \wh X_0$ be the blow-up of $\wh X_0$ 
at $P_0$.  Its closed fiber consists of two copies of the projective 
line over $k$: the proper transform of the closed fiber $X_0$ of $\wh 
X_0$, and the exceptional divisor $E$, which is the projective $y$-line over $k$.  
Here $y = (x - \alpha)/t$, and $y= \infty$ at the point where $X_0$ and $E$ meet.
Let $\wh Y \to \wh X$ be the blow-down that contracts the 
proper transform of $X_0$ to a point; thus $\wh X$ is the projective 
$y$-line over $k[[t]]$.  Over the closed point $P$ on $\wh X$ given by 
$(q(y),t)$, there is a unique closed point $Q$ on $\wh Y$ (since $y \ne \infty$ at $P$), 
and the morphism $\wh Y \to \wh X$ induces an isomorphism of $F_P$ with $F_Q$. 
As in the proof of
\cite[Proposition~3.2.4]{HHK:ref}, there is an inclusion of fields 
$F_{P_0} \hookrightarrow F_Q$.  (Namely, the morphism $\Spec(\wh R_Q) 
\to \wh Y$ factors through $\wh Y_{P_0}$, the pullback of $\wh Y$ over 
$\wh X_0$ with respect to
$\Spec(\wh R_{P_0}) \to \wh X_0$; and $\wh Y_{P_0} \to \Spec(\wh 
R_{P_0})$ is a birational isomorphism since $\wh Y \to \wh X_0$ is.) 
Since $F_Q \cong F_P$ contains $k'$, this yields a differential 
embedding of $R$ into $F_P$ (as in the above proof), after which we may apply 
Theorem~\ref{thm BHHW}
to a Picard-Vessiot ring $R_1$ over $F_P$ in order to conclude the argument.

\begin{rem}\label{nonsplit}
In Theorem~\ref{main thm}, one cannot expect to be able to drop the 
hypothesis that the embedding
problem is split, even in the case of finite constant groups.  This is 
due to the fact that $k(x)$
has cohomological dimension greater than one if $k$ is not algebraically 
closed, and hence general finite embedding problems need not have a proper
solution by \cite[I.3.4, 
Proposition~16]{Serre:CG}.  As a concrete example, take the 
$\Z/2\Z$-extension of $\Q(x)$ given by $y^2=x$; this is a Picard-Vessiot 
ring for the differential equation $y'=\frac{1}{2x}y$. This extension and the 
(non-split) short exact sequence $0 \to 2\Z/4\Z \to \Z/4\Z \to \Z/2\Z 
\to 0$ define a differential embedding problem over $\Q(x)$, which 
induces a differential embedding problem over $\Q((t))(x)$.  Suppose 
that the latter problem has a proper solution $E/\Q((t))(x)$; this is a 
$\Z/4\Z$-Galois extension of $\Q((t))(x)$ in which $\Q((t))$ is 
algebraically closed, and which contains $\Q((t))(y)$.  Let $A$ be the 
integral closure of $\Z[[t]][x]$ in $E$, and let $e$ be the ramification 
index over $(t)$ (so $e$ equals $1$ or $2$).  After replacing $t$ by 
$s=t^{1/e}$, we may assume that $A$ is unramified over $(t)$.  So the 
normalization $A_0$ of $A/(t)$ is a
$\Z/4\Z$-\'etale algebra over $\Q(x)$ whose inertia group at $(x)$ is 
all of $\Z/4\Z$, since it surjects onto $\Z/2\Z$.  Thus $A_0$ is a 
domain; and the base change of its fraction field to $\Q((x))$ is a 
$\Z/4\Z$-Galois field extension $L/\Q((x))$ that is totally ramified, 
has no residue field extension, and contains $\Q((y))$.  Hence $L$ is 
obtained by adjoining to $\Q((y))$ a square root $z$ of some element, 
which up to a square is of the form $cy$ for some $c \in \Q^\times$.
But since $L/\Q((x))$ is $\Z/4\Z$-Galois and dominates $\Q((y))$, the 
generator of its Galois group lifts the automorphism $y \mapsto -y$ of 
$\Q((y))$ and so takes $z$ to a square root of $-cy$.  But since $L/\Q((x))$ has no residue field extension, $L$ does not contain a square root of $-1$, a contradiction.
\end{rem}

\begin{rem}
By using Theorem~\ref{main thm} above as the key ingredient, the companion paper
\cite{BHHP} proves that every split differential embedding problem over $K(x)$ has a proper solution if $K$ is any large field (in the sense of \cite{pop:littlesurvey}) of infinite transcendence degree over~$\Q$.  Since
$k((t))$ is such a field $K$ for any field $k$ of characteristic zero, we obtain the following conclusion, extending Theorem~\ref{main thm} (\cite{BHHP}, Theorem~4.3(b)):

{\em Let $k$ be a field of characteristic zero and let $K=k((t))$. Consider the rational function field $F=K(x)$ with derivation $\del=\frac{d}{dx}$. Then every split differential embedding problem over $F$ has a proper solution.} 
\end{rem}

\medskip

\noindent Author information:

\medskip

\noindent Annette Bachmayr, n\'ee Maier: Mathematisches Institut der Universit\"at Bonn,
D-53115 Bonn, Germany.\\ email: {\tt bachmayr@math.uni-bonn.de}

\medskip

\noindent David Harbater: Department of Mathematics, University of Pennsylvania, Philadelphia, PA 19104-6395, USA.\\ email: {\tt harbater@math.upenn.edu}

\medskip

\noindent Julia Hartmann:  Department of Mathematics, University of Pennsylvania, Philadelphia, PA 19104-6395, USA.\\ email: {\tt hartmann@math.upenn.edu}

\end{document}